\documentclass[a4paper,10pt]{amsart}
\usepackage[english]{babel}
\usepackage{amssymb}
\usepackage{hyperref}
\usepackage[initials]{amsrefs}

\usepackage[all]{xy}
\SelectTips{cm}{}
\usepackage{tikz}
\usepackage{float}

\newcommand{\bbZ}{{\mathbb Z}}

\newcommand{\calN}{\mathcal{N}}

\newtheorem{theorem}{Theorem}[section]
\newtheorem{lemma}[theorem]{Lemma}

\newtheorem{corollary}[theorem]{Corollary}
\newtheorem{proposition}[theorem]{Proposition}

\theoremstyle{definition}
\newtheorem{definition}[theorem]{Definition}

\newtheorem{example}[theorem]{Example}

\newtheorem{remark}[theorem]{Remark}

\begin{document}

\title{L$^2$-invisibility and a class of local similarity groups}

\author{Roman Sauer}
\address{Karlsruhe Institute of Technology, Karlsruhe, Germany}
\email{roman.sauer@kit.edu}

\author{Werner Thumann}
\address{Karlsruhe Institute of Technology, Karlsruhe, Germany}
\email{werner.thumann@kit.edu}

\thanks{Both authors gratefully acknowledge support by the DFG grant 1661/3-1.}


\begin{abstract}
	In this note we show that the members of a certain class of local similarity groups are $l^2$-invisible,
	i.e.~the (non-reduced) group homology of the regular unitary representation vanishes in all degrees. This 
	class contains groups of type $F_{\infty}$, e.g.~Thompson's group~$V$ and Nekrashevych-R\"over groups. They 
	yield counterexamples to a generalized zero-in-the-spectrum conjecture for groups of type $F_{\infty}$.
\end{abstract}
\maketitle

\section{Introduction}\label{sec: intro}

The \emph{zero-in-the-spectrum conjecture} (or question) appears for the first time in Gromov's article 
\cite{gromov-riemannian}. It states that for an aspherical closed Riemannian manifold~$M$ there always exists
$p\geq 0$ such that the spectrum of the Laplacian $\Delta_p$ acting on the square integrable $p$-forms on the 
universal covering of $M$ contains zero. The latter is equivalent (see~\cite{lott}) to the group-homological 
statement 
\begin{equation}
	\exists_{p\ge 0}~H_p\bigl(\Gamma, \calN(\Gamma)\bigr)\ne 0, 
\end{equation}
where $\Gamma=\pi_1(M)$ and $\calN(\Gamma)$ is the group von Neumann algebra of~$\Gamma$. The zero-in-the-spectrum 
conjecture is motivated and implied by the strong Novikov 
conjecture~\citelist{\cite{gromov-riemannian}*{4.B.}\cite{lück-book}*{Theorem~12.7 on p.~443}}. We call a group 
$\Gamma$ \emph{$l^2$-invisible} if 
\[\forall_{p\ge 0}~H_p\bigl(\Gamma, \calN(\Gamma)\bigr)=0. \] 
By~\cite{lück-book}*{Lemmas~6.98 on p.~286 and~12.3 on p.~438} a group $\Gamma$ of type $F_\infty$ is 
$l^2$-invisible iff 
\[\forall_{p\ge 0}~ H_p\bigl(\Gamma, l^2(\Gamma)\bigr)=0.\]
Note that $l^2$-invisibility of a group is a much stronger property than the vanishing of its $l^2$-Betti numbers,
which is equivalent to the vanishing of the reduced homology.

A more general zero-in-the-spectrum question by Lott~\cite{lott}, where one drops the asphericity condition, was 
answered in the negative by Farber and Weinberger~\cite{farber-weinberger}. This note is concerned with an 
algebraic generalization of the zero-in-the-spectrum question, which was raised -- in different terminology -- by 
L\"uck~\citelist{\cite{durham}*{Remark~12.16}\cite{lück-book}*{Remark~12.4 on p.~440}}:  
\begin{equation}\label{eq: algebraic question}
	\text{\emph{Are there groups of type $F$ or $F_\infty$ that are $l^2$-invisible?}}
\end{equation}
Recall that a group $G$ is of type $F$ iff there is model of the classifying space $\mathrm{B}G$ with finitely 
many cells and of type $F_\infty$ iff there is a model of $\mathrm{B}G$ with finitely many cells in each 
dimension. Without any finiteness condition on the group, $l^2$-invisible groups are easily constructed by 
taking suitable infinite products (see \emph{loc.~cit.}). In the spirit of the zero-in-the-spectrum conjecture 
one might expect the answer to~\eqref{eq: algebraic question} to be negative, but, in fact, we provide here many 
examples of $F_\infty$-groups that are $l^2$-invisible. 

Hughes~\cite{hughes-haagerup} introduced a certain class of groups acting on compact ultrametric spaces which we 
call \emph{local similarity groups} for short (see Section~\ref{sec: local sim groups} for details). Assuming 
there are only finitely many $\mathrm{Sim}$-equivalence classes of balls and the similarity structure satisfies
a condition called \emph{rich in ball contractions}, these groups satisfy property $F_\infty$~\cite{farley-hughes}.
In Section \ref{sec: dual contraction} we will introduce another property for similarity structures, called 
\emph{dually contracting}, which is implied by rich in ball contractions and enables us to prove the following 
theorem in Section~\ref{sec: proof}.

\begin{theorem}\label{thm: main thm}
	Let $X$ be a compact ultrametric space with similarity structure $\mathrm{Sim}$. If $\mathrm{Sim}$ is dually
	contracting, then the local similarity group $\Gamma=\Gamma(\mathrm{Sim})$ is $l^2$-invisible. 
\end{theorem}

The well known Thompson group $V$ can be realized as a local similarity group which is contained in this class, 
as well as the Nekrashevych-R\"over groups $V_d(H)$ (Example~\ref{ex: thompson et al}). They are also of type 
$F_\infty$ by the results in \cite{farley-hughes}. Already Brown showed in \cite{brown-finiteness}*{Theorem~4.17} 
that $V$ is of type $F_\infty$. Unfortunately, we cannot say anything about the $F$-part of 
question~\eqref{eq: algebraic question} since the groups we consider here are easily seen to have infinite 
cohomological dimension. Indeed, as a byproduct of our argument, we obtain the following statement which implies 
infinite cohomological dimension \cite{brown-book}*{Prop.~(6.1) on p.~199 and (6.7) on p.~202}.

\begin{theorem}\label{thm: byproduct thm}
	Let $X$ be a compact ultrametric space with similarity structure $\mathrm{Sim}$. If there are only finitely 
	many $\mathrm{Sim}$-equivalence classes of balls and $\mathrm{Sim}$ is rich in ball contractions, then the 
	local similarity group $\Gamma=\Gamma(\mathrm{Sim})$ satisfies $H^\ast(\Gamma, \bbZ[\Gamma])=0$ in all degrees.  
\end{theorem}

Note that the case $\Gamma=V$ has already been treated in \cite{brown-finiteness}*{Theorem~4.21}.

\vspace{1mm}
{\it Related work}: In \cite{oguni} Oguni defines an algorithm which takes a finitely presented non-amenable 
group $G$ as input and gives a finitely presented group $G_\Psi$ with $H_p(G_\Psi,\mathcal{N}(G_\Psi))=0$ for 
all $p$. But it is not known when $G_\Psi$ is of type $F_\infty$. In fact, $G_\Psi$ is not even $F_3$ if $G$ is
a free group.

\section{Local similarity groups}\label{sec: local sim groups}

In this section we review the definition and fix the terminology for Hughes' class of local similarity groups.

Recall that	an \emph{ultrametric space} is a metric space $(X,d)$ such that
\[d(x,y)\leq\max\{d(x,z),d(z,y)\}~~\text{for all $x,y,z\in X$}.\]
In this paper, \emph{$X$ always denotes a compact ultrametric space}. The endspace of a rooted proper 
$\mathbb{R}$-tree is a compact ultrametric space and, conversely, every compact ultrametric space of diameter 
less than or equal to one is the endspace of a rooted proper $\mathbb{R}$-tree. See \cite{hughes-trees} for more 
information in this direction. By a ball in $X$, we always mean a subset of the form
\[B(x,r)=\{y\in X\ |\ d(x,y)\leq r\}\]
with $x\in X$ and $r\geq 0$. Two balls are always either disjoint or one contains the other. A non-empty subset 
is open and closed if and only if it is a union of finitely many balls. Let $X,Y$ be compact ultrametric spaces. 
A homeomorphism $\gamma:X\rightarrow Y$ is called 
\begin{itemize}
	\item an \emph{isometry} iff $d(\gamma(x_1),\gamma(x_2))=d(x_1,x_2)$ for all $x_1,x_2\in X$.
	\item a \emph{similarity} iff there is a $\lambda>0$ with $d(\gamma(x_1),\gamma(x_2))=\lambda d(x_1,x_2)$ for all $x_1,x_2\in X$.
	\item a \emph{local similarity} iff for every $x\in X$ there are balls $A\subset X$ and $B\subset Y$ with $x\in A$, $\gamma(x)\in B$ and $\gamma|_A:A\rightarrow B$ is a similarity.
\end{itemize}
The set of all local similarities on $X$ forms a group and is denoted by $LS(X)$.

\begin{definition}[\cite{hughes-haagerup}*{Definition~3.1}]
	Let $X$ be a compact ultrametric space. A \emph{similarity structure} $\mathrm{Sim}$ on $X$ (called 
	\emph{finite similarity structure} in~\citelist{\cite{farley-hughes}\cite{hughes-haagerup}}) consists of a 
	finite set $\mathrm{Sim}(B_1,B_2)$ of similarities $B_1\rightarrow B_2$ for every ordered pair of balls 
	$(B_1,B_2)$ such that the following axioms are satisfied: 
	\begin{itemize}
		\item (Identities) Each $\mathrm{Sim}(B,B)$ contains the identity.
		\item (Inverses) If $\gamma\in\mathrm{Sim}(B_1,B_2)$ then also $\gamma^{-1}\in\mathrm{Sim}(B_2,B_1)$.
		\item (Compositions) If $\gamma_1\in\mathrm{Sim}(B_1,B_2)$ and $\gamma_2\in\mathrm{Sim}(B_2,B_3)$ then 
			also $\gamma_2\circ\gamma_1\in\mathrm{Sim}(B_1,B_3)$.
		\item (Restrictions) If $\gamma\in\mathrm{Sim}(B_1,B_2)$ and $B_3\subset B_1$ is a subball then also 
			$\gamma|_{B_3}\in\mathrm{Sim}(B_3,\gamma(B_3))$.
	\end{itemize}
\end{definition}

A local similarity $\gamma:X\rightarrow X$ is locally determined by $\mathrm{Sim}$ iff for every $x\in X$ there 
is a ball $x\in B\subset X$ such that $\gamma(B)$ is a ball and $\gamma|_B\in\mathrm{Sim}(B,\gamma(B))$. The set 
of all local similarities locally determined by $\mathrm{Sim}$ forms a group, denoted by $\Gamma(\mathrm{Sim})$, 
and is called the \emph{local similarity group associated to $(X,\mathrm{Sim})$}. A group arising this way is 
called a \emph{local similarity group}.

\begin{example}[cf.~\cite{hughes-haagerup}*{Section~4}]\label{ex: thompson et al}
	We recall the alphabet terminology of the rooted $d$-ary tree. Let $A=\{a_1,...,a_d\}$ be a set of $d$ 
	letters. A word in $A$ is just an element of $A^n$ for some $n\geq 1$ or the empty word. An infinite word is 
	an element in the countable product $A^{\omega}=\prod_{\mathbb{N}}A$. The simplicial tree associated to $A$ 
	has words as vertices and an edge between to words $v,w$ iff there is an $x\in A$ with $vx=w$ or $v=wx$. The 
	root is the empty word. The endspace of this tree can be identified with the set of infinite words. It comes 
	with a natural ultrametric defined by
	\[d(x,y):=
		\begin{cases}
		0&\text{if }x=y\\
		\exp(1-n)&\text{if }n=\min\{k\ |\ x_k\neq y_k\}
		\end{cases}\]
	where $x=x_1x_2...$ and $y=y_1y_2...$ are infinite words. Since the tree is locally finite, the endspace with 
	this metric is compact. Call it $X$.\\
	Now let $H$ be a subgroup of the symmetric group $\Sigma_d$ of $A$. Define a similarity structure 
	$\mathrm{Sim}$ on $X$ as follows. If $B_1$ and $B_2$ are balls of $X$ then there are unique words $w_1$
	 and $w_2$ such that $B_1=w_1A^{\omega}$ and $B_2=w_2A^{\omega}$. If $\sigma\in H$ then 
	\[\gamma_\sigma:w_1A^{\omega}\rightarrow w_2A^{\omega}\hspace{5mm}w_1x_1x_2...\mapsto w_2\sigma(x_1)\sigma(x_2)...\]
	defines a similarity $B_1\rightarrow B_2$. Set $\mathrm{Sim}(B_1,B_2):=\{\gamma_\sigma\ |\ \sigma\in H\}$. 
	This defines a similarity structure $\mathrm{Sim}$ on $X$. The corresponding local similarity group is the 
	Nekrashevych-R\"over group $V_d(H)$ considered in \cite{nekrashevych} and \cite{röver}. In the case $H=1$, 
	this specializes to the Higman-Thompson groups $V_d$ and in particular to the well known Thompson group $V$ 
	for $d=1$.
\end{example}

If $A$ and $B$ are balls in $X$, we say that $A$ and $B$ are \emph{$\mathrm{Sim}$-equivalent} iff there exists a 
similarity $A\rightarrow B$ in $\mathrm{Sim}$. Denote by $[A]$ the corresponding equivalence class of $A$. More 
generally, if $Y$ and $Z$ are non-empty closed open subspaces of $X$, we say that $Y$ and $Z$ are 
\emph{locally $\mathrm{Sim}$-equivalent} iff there exists a local similarity $g:Y\rightarrow Z$ locally 
determined by $\mathrm{Sim}$. This means that for each $y\in Y$ there is a ball $B$ of $X$ with $y\in B$ and 
$B\subset Y$ such that $g(B)$ is a ball of $X$ with $g(B)\subset Z$ and $g|_B\in\mathrm{Sim}(B,g(B))$. Denote by
$\langle Y\rangle$ the corresponding equivalence class of $Y$. Of course, two balls are locally 
$\mathrm{Sim}$-equivalent if they are $\mathrm{Sim}$-equivalent.

If $Y\subset X$ is a non-empty closed open subspace, then one can restrict the similarity structure 
$\mathrm{Sim}$ to one on $Y$ by defining
\[\mathrm{Sim}|_Y:=\{\gamma\in\mathrm{Sim}\ |\ \mathrm{dom}(\gamma)\cup\mathrm{codom}(\gamma)
	\subset Y\}\cup\{\mathrm{id}_B\ |\ B\text{ a ball of }Y\}\]
where $\mathrm{dom}(\gamma)$ is the domain of $\gamma$ and $\mathrm{codom}(\gamma)$ is the codomain of $\gamma$. 
Note that a ball of $Y$ need not be a ball of $X$, so we have to add the identity maps in the definition of 
$\mathrm{Sim}|_Y$. The group $\Gamma(\mathrm{Sim}|_Y)$ is a subgroup of $\Gamma(\mathrm{Sim})$. More precisely, 
$\Gamma(\mathrm{Sim}|_Y)$ is isomorphic to the subgroup of $\Gamma(\mathrm{Sim})$ consisting of the elements 
$\alpha:X\rightarrow X$ with $\alpha(x)=x$ for $x\in X\setminus Y$. The proof of the next lemma is easy and left 
to the reader.

\begin{lemma}\label{lem: restricted subgroups}
	Let $X$ be a compact ultrametric space and $\mathrm{Sim}$ a similarity structure on $X$. Let $Y,Z\subset X$ 
	be two non-empty closed open subspaces with $\langle Y\rangle=\langle Z\rangle$. Then the groups 
	$\Gamma(\mathrm{Sim}|_Y)$ and $\Gamma(\mathrm{Sim}|_Z)$ are isomorphic.
\end{lemma}

Let $X$ be a compact ultrametric space. There is a rooted locally finite simplicial tree associated to $X$, 
called the ball hierarchy. It has balls of $X$ as vertices and an edge between the balls $A$ and $B$ whenever 
$A$ is a proper maximal subball of $B$ or vice versa. Take the ball $X$ as root. It is locally finite because 
$X$ is compact. Now define the \emph{depth} of a ball $B$ in $X$, denoted by $\mathrm{depth}(B)$, to be the 
distance between the vertex $B$ and the root $X$ in the ball hierarchy tree. We will need the following lemma 
in Section~\ref{sec: dual contraction}.

\begin{lemma}\label{lem: depth and partition}
	Let $X$ be a compact ultrametric space and $\mathcal{P}$ a partition of $X$ into non-empty closed open 
	subspaces, i.e.~$\mathcal{P}$ is a finite set of pairwise disjoint non-empty closed open subspaces of $X$ 
	so that the union of the elements of $\mathcal{P}$ is all of $X$. Then there exists $N\in\mathbb{N}$ such 
	that every ball with depth at least $N$ is contained in some $P\in\mathcal{P}$.
\end{lemma}

\begin{proof}
	Since every non-empty closed open subspace in a compact ultrametric space is a finite union of balls, we can 
	assume without loss of generality that each $P\in\mathcal{P}$ is a ball. Set
	\[N:=\max\{\mathrm{depth}(P)\ |\ P\in\mathcal{P}\}.\]
	We claim that every ball with depth at least $N$ is contained in some $P\in\mathcal{P}$. Assume the 
	contradiction. Then there exists a ball $B$ such that $\mathrm{depth}(B)\geq\mathrm{depth}(P)$ for all 
	$P\in\mathcal{P}$ but $B\not\subset P$ for all $P\in\mathcal{P}$. The latter means that for each 
	$P\in\mathcal{P}$ either $B\cap P=\emptyset$ or $P\subsetneq B$. But $P\in\mathcal{P}$ cannot be a proper 
	subball of $B$ because of the depth condition. So we have $B\cap P=\emptyset$ for all $P\in\mathcal{P}$ 
	which contradicts $X=\bigcup_{P\in\mathcal{P}}P$.
\end{proof}

\begin{definition}
	We call $\gamma:A\rightarrow B$ in a similarity structure $\mathrm{Sim}$
	\begin{itemize}
		\item \emph{contracting} iff $A\subsetneq B$ or $B\subsetneq A$.
		\item \emph{separating} iff $A\cap B=\emptyset$.
		\item \emph{equalizing} iff $A=B$.
	\end{itemize}
\end{definition}

In general, the precise relationship between the similarity structure and the corresponding local similarity 
group is not yet understood very well. The following two propositions are easy results in this direction.

\begin{proposition}\label{prop: finiteness}
	Let $X$ be a compact ultrametric space and $\mathrm{Sim}$ a similarity structure on $X$. Then the following 
	are equivalent.
	\begin{itemize}
		\item[i)] $\Gamma(\mathrm{Sim})$ is finite.
		\item[ii)] There are only finitely many separating elements in $\mathrm{Sim}$.
		\item[iii)] There are only finitely many non-identity elements in $\mathrm{Sim}$.
	\end{itemize}
	In this case, $\mathrm{Sim}$ contains no contracting elements and $\Gamma(\mathrm{Sim})$ fixes all points of 
	$X$ except a finite subset of isolated points. It permutes these isolated points in a way such that 
	$\Gamma(\mathrm{Sim})\cong\Sigma_{d_1}\times...\times\Sigma_{d_n}$ is a finite product of finite symmetric 
	groups. 
\end{proposition}

\begin{proof}
	First we make a series of observations.
	
	\vspace{2mm}	
	{\it Observation 1}: If $\gamma:A\rightarrow B$ is a separating element in $\mathrm{Sim}$, then we can 
	construct an element $\alpha\in\Gamma(\mathrm{Sim})$ by defining 
	$\alpha|_A=\gamma$ and $\alpha|_B=\gamma^{-1}$ and $\alpha(x)=x$ for all other elements $x\in X$. If 
	$\gamma_i:A_i\rightarrow B_i$ with $i=1,2$ are two separating elements and $\gamma_1\neq\gamma_2$, then the 
	corresponding $\alpha_i$ also satisfy $\alpha_1\neq\alpha_2$. In particular, if there are infinitely many 
	distinct such $\gamma_i$, then $\Gamma(\mathrm{Sim})$ is infinite.
	
	\vspace{2mm}
	{\it Observation 2}: Assume there is a contracting element $\gamma:A\rightarrow B$ in $\mathrm{Sim}$. Assume 
	without loss of generality $B\subsetneq A$. Then there are infinitely many distinct separating elements in 
	$\mathrm{Sim}$ which can be constructed as follows. Let $C$ be a ball in $A\setminus B$ and define 
	$C_i=\gamma^i(C)$ for $i\in\{0,1,2,...\}$. Fix some $i$. Observe $\gamma^{i+k}(C)\subset\gamma^i(B)$ for all 
	$k\geq 1$ and $\gamma^i(C)\cap\gamma^i(B)=\emptyset$. It follows $C_i\cap C_{i+k}=\emptyset$ for all 
	$k\geq 1$. Therefore, we can define $\gamma_i:=\gamma|_{C_i}:C_i\rightarrow\gamma(C_i)$ for 
	$i\in\{0,1,2,...\}$ and obtain an infinite sequence of distinct separating elements in $\mathrm{Sim}$.
	
	\vspace{2mm}
	{\it Observation 3}: Assume there is a separating element $\gamma:A\rightarrow B$ in $\mathrm{Sim}$ with 
	$A$ being an infinite set. Then $A$ has infinitely many subballs and we see at once that there are 
	infinitely many distinct separating elements in $\mathrm{Sim}$.
	
	\vspace{2mm}
	{\it Observation 4}: Assuming we only have finitely many separating elements in $\mathrm{Sim}$, then we 
	claim that there are only finitely many non-identity equalizing elements in $\mathrm{Sim}$. This follows if 
	we show that each $\gamma:C\rightarrow C$ in $\mathrm{Sim}$ (which is an isometry) is itself locally 
	determined by identities and separating elements. This means that for each $x\in C$ we find a ball 
	$D\subset C$ with $x\in D$ and either $\gamma|_D=\mathrm{id}_D$ or $D\cap\gamma(D)=\emptyset$. We start by 
	noting that for any isometry $\alpha:Y\rightarrow Y$ of a compact ultrametric space $Y$, if 
	$\alpha\neq\mathrm{id}_Y$, then there must be a ball $D\subset Y$ such that $\alpha(D)\cap D=\emptyset$. 
	Now consider the maximal proper subballs of $C$. Either $\gamma$ is the identity on such a ball $B$ or 
	$\gamma$ maps $B$ to another such ball or $\gamma$ maps $B$ to itself and is not the identity. Only in the 
	last case we have to go a step deeper and consider the maximal proper subballs of $B$. Since 
	$\gamma|_B\neq\mathrm{id}_B$, we know that there must be a subball $E\subset B$ such that 
	$\gamma(E)\cap E=\emptyset$. Since there are only finitely many separating elements in $\mathrm{Sim}$, 
	we see that this process has to stop at some point. This proves the claim.
	
	\vspace{2mm}
	{\it i) $\Rightarrow$ ii)}: This is clear from the first observation.
	
	\vspace{2mm}
	{\it ii) $\Rightarrow$ iii)}: $\mathrm{Sim}$ cannot contain any contracting elements because of the second 
	observation. Because of the fourth observation, there are also only finitely many non-identity equalizing 
	elements. So $\mathrm{Sim}$ has only finitely many non-identity elements.
	
	\vspace{2mm}
	{\it iii) $\Rightarrow$ i)}: This is clear from the definition of $\Gamma(\mathrm{Sim})$.
	
	\vspace{2mm}
	Now we turn to the last statements. The first of these follows from the second observation. From the fourth 
	observation we know that each element in $\Gamma(\mathrm{Sim})$ is locally determined by identities and 
	separating elements in $\mathrm{Sim}$. We know that there are only finitely many elements of the latter 
	type in $\mathrm{Sim}$. From the third observation we deduce that these separating elements can only be 
	defined on finite subballs (which consist of finitely many isolated points). This proves that 
	$\Gamma(\mathrm{Sim})$ fixes all points of $X$ except possibly a finite subset $Y\subset X$ of isolated 
	points. Since $\Gamma(\mathrm{Sim}|_Y)\cong\Gamma(\mathrm{Sim})$, we can assume without loss of generality 
	that $X$ itself contains only finitely many points. In this case, by the restriction property of a 
	similarity structure, each element in $\Gamma(\mathrm{Sim})$ is locally determined by similarities in 
	$\mathrm{Sim}$ of the form $A\rightarrow B$ where $A$ and $B$ are singleton balls. The definition
	\enlargethispage{\baselineskip}
	\[x\sim y\hspace{5mm}:\Longleftrightarrow\hspace{5mm}\mathrm{Sim}\big(\{x\},\{y\}\big)\neq\emptyset\]
	gives an equivalence relation on $X$. Let $X_1,...,X_n$ be the corresponding equivalence classes. We have
	\[\Gamma(\mathrm{Sim})\cong\Gamma(\mathrm{Sim}|_{X_1})\times...\times\Gamma(\mathrm{Sim}|_{X_n})\]
	and $\Gamma(\mathrm{Sim}|_{X_i})\cong\Sigma_{d_i}$ where $d_i$ is the number of elements in $X_i$. 
	This proves the last claim of the proposition.
\end{proof}

\begin{proposition}\label{prop: loc fin}
	Let $X$ be a compact ultrametric space and $\mathrm{Sim}$ a similarity structure on $X$ such that 
	$\mathrm{Sim}(B_1,B_2)=\emptyset$ whenever $\mathrm{depth}(B_1)\neq\mathrm{depth}(B_2)$. Then the local 
	similarity group $\Gamma=\Gamma(\mathrm{Sim})$ is locally finite.
\end{proposition}

\begin{proof}
	First let $\alpha$ be an arbitrary element in $\Gamma$. For each $x\in X$ let $A_x$ be the maximal ball 
	with $x\in A_x$ such that there is an element $\alpha_x\in\mathrm{Sim}(A_x,\alpha(A_x))$ and 
	$\alpha|_{A_x}=\alpha_x$. The set of balls $\{A_x\ |\ x\in X\}$ is a partition of $X$ called the partition 
	into maximum regions for $\alpha$. Define
	\[\mathrm{depth}(\alpha):=\max\{\mathrm{depth}(A_x)\ |\ x\in X\}\]
	Note that from the assumption on $\mathrm{Sim}$, each similarity in $\mathrm{Sim}$ preserves the depth of 
	balls. Let $\alpha,\beta\in\Gamma$ and let $\mathcal{P}$ and $\mathcal{Q}$ be the corresponding partitions 
	into maximum regions. Observe the composition $\beta\circ\alpha$. It is locally determined by $\mathrm{Sim}$ 
	on a partition $\mathcal{R}$ of $X$ into balls such that for every $R\in\mathcal{R}$ either $R=P$ for some 
	$P\in\mathcal{P}$ or $R=(\alpha|_P)^{-1}(Q)$ for some $P\in\mathcal{P}$ and some $Q\in\mathcal{Q}$ with 
	$Q\subset\alpha(P)$. It follows
	\[\mathrm{depth}(\beta\circ\alpha)\leq\max\{\mathrm{depth}(R)\ |\ R\in\mathcal{R}\}\leq\max\{\mathrm{depth}(\alpha),\mathrm{depth}(\beta)\}.\]
	So if $\alpha_1,...,\alpha_k\in\Gamma$, we also have
	\begin{equation}\label{eq: proof loc fin}
		\mathrm{depth}(\alpha_1\circ...\circ\alpha_k)\leq\max\{\mathrm{depth}(\alpha_1),...,
		\mathrm{depth}(\alpha_k)\}.
	\end{equation}
	Now let $\Lambda$ be a subgroup of $\Gamma$ with finite generating set $\gamma_1,...,\gamma_n$. From 
	\eqref{eq: proof loc fin} we deduce that 
	\[\mathrm{depth}(\lambda)\leq\max\{\mathrm{depth}(\gamma_i)\ |\ i=1,...,n\}=:N\]
	for each $\lambda\in\Lambda$. We claim that there are only finitely many local similarities $\gamma$ 
	locally determined by $\mathrm{Sim}$ such that $\mathrm{depth}(\gamma)\leq N$. This follows because there 
	are only finitely many balls $B$ with $\mathrm{depth}(B)\leq N$ and each $\mathrm{Sim}(B_1,B_2)$ is finite 
	by definition.
\end{proof}

\section{A condition on the similarity structure}\label{sec: dual contraction}

Here we introduce a condition on similarity structures used for the proof of Theorem \ref{thm: main thm} and 
Theorem \ref{thm: byproduct thm}.

\begin{definition}\label{def: dual contraction}
	Let $X$ be a compact ultrametric space and $\mathrm{Sim}$ a similarity structure on $X$. We say 
	$\mathrm{Sim}$ is {\it dually contracting} or has a {\it dual contraction} if there are two disjoint proper 
	subballs $B_1$ and $B_2$ of $X$ together with similarities $X\rightarrow B_1$ and $X\rightarrow B_2$ in 
	$\mathrm{Sim}$.
\end{definition}

\begin{remark}
	The property in Definition \ref{def: dual contraction} is rather a property of the similarity structure 
	than of the local similarity group $\Gamma(\mathrm{Sim})$. To illustrate the precise meaning of this 
	statement, consider the following. Let $X$ be a compact ultrametric space and $\mathrm{Sim}$ a similarity 
	structure on it. Remove all elements in $\mathrm{Sim}$ of the form $A\rightarrow B$ where either $A=X\neq B$ 
	or $A\neq X=B$. Denote the remaining set of similarities by $\mathrm{Sim}^-$. It is easy to see that 
	$\mathrm{Sim}^-$ still forms a similarity structure on $X$. Furthermore, since no similarity in 
	$\mathrm{Sim}\setminus\mathrm{Sim}^-$ can be used to form a local similarity on $X$, the groups 
	$\Gamma(\mathrm{Sim})$ and $\Gamma(\mathrm{Sim}^-)$ are the same as sets of local similarities on $X$. 
	However, even if $\mathrm{Sim}$ is dually contracting, the similarity structure $\mathrm{Sim}^-$ never is. 
	But it can be extended to a dually contracting one. We therefore call a similarity structure 
	{\it potentially} dually contracting if it can be extended in such a way that the corresponding local 
	similarity groups are the same (as sets of local similarities) and the extension is dually contracting.
\end{remark}

\begin{example}
	The similarity structures presented in Example \ref{ex: thompson et al} are dually contracting. So Theorem 
	\ref{thm: main thm} applies to the Nekrashevych-R\"over groups $V_d(H)$ and in particular to the Thompson 
	group $V$.
\end{example}

\begin{example}
	If $X$ is a compact ultrametric space and $\mathrm{Sim}$ a similarity structure on $X$ such that the local 
	similarity group $\Gamma(\mathrm{Sim})$ is finite, then, by Proposition \ref{prop: finiteness}, 
	$\mathrm{Sim}$ cannot be potentially dually contracting.
\end{example}

\begin{example}
	Let $X$ be the end space of the rooted binary tree with the usual order. Let $B_1$ and $B_2$ be the two 
	maximal proper subballs of $X$. Let $\mathrm{Sim}$ be the similarity structure generated by the unique order 
	preserving similarity $\gamma:B_1\rightarrow B_2$, i.e.~the smallest similarity structure on $X$ containing 
	$\gamma$. More precisely, the non-trivial similarities in $\mathrm{Sim}$ are the unique order preserving 
	similarites $xA^\omega\rightarrow\bar{x}A^\omega$ where $\bar{x}$ is obtained from $x$ by changing the first 
	letter of $x$, e.g.~$\bar{x}=101$ if $x=001$. It follows from Proposition \ref{prop: finiteness} that 
	$\Gamma=\Gamma(\mathrm{Sim})$ is infinite and from Proposition \ref{prop: loc fin} that $\Gamma$ is locally 
	finite. It is therefore not finitely generated. Since it is locally finite, it is also elementary amenable 
	and consequently $H_0(\Gamma,\mathcal{N}(\Gamma))\neq 0$. This shows that we cannot drop the condition 
	dually contracting in Theorem \ref{thm: main thm}.\\
	However, this similarity structure is not potentially dually contracting. Otherwise there would be a 
	similarity $\delta:X\rightarrow A$ with $A$ a proper subball of $X$. Let $C$ be another proper subball of 
	$X$ with $A\cap C=\emptyset$. We have $C\cap\delta(C)=\emptyset$. The restriction 
	$\hat{\delta}=\delta|_C:C\rightarrow\delta(C)=:D$ fits into a local similarity on $X$. Just define 
	$\alpha:X\rightarrow X$ by
	\[\alpha|_C:=\hat{\delta}\hspace{5mm}\alpha|_D:=\hat{\delta}^{-1}\hspace{5mm}\alpha|_{X\setminus(C\cup D)}:=
		\mathrm{id}.\]
	This local similarity $\alpha$ is not in $\Gamma(\mathrm{Sim})$ because neither $\hat{\delta}$ nor any of 
	its restrictions is an element of $\mathrm{Sim}$.
\end{example}

\begin{remark}
	In \cite{farley-hughes}, Farley and Hughes introduced a condition on a similarity structure $\mathrm{Sim}$, 
	called {\it rich in ball contractions}, which is defined as follows. There exists a constant $c>0$ such that 
	for every $k\geq c$ and $(B_1,...,B_k)$ a $k$-tuple of balls, there is a ball $B$ with at least two maximal 
	proper subballs and an injection
	\[\sigma:\{A\ |\ A\text{ maximal proper subball of }B\}\rightarrow\{(B_i,i)\ |\ 1\leq i\leq k\}\]
	with $[A]=[B_i]$ whenever $\sigma(A)=(B_i,i)$. In \emph{loc.~cit.} it is shown that local similarity groups 
	arising from similarity structures having this property and with only finitely many $\mathrm{Sim}$-equivalence 
	classes of balls are of type $F_\infty$. It is quite clear that rich in ball contractions implies dually 
	contracting, just take $(X,...,X)$ as a $k$-tuple of balls. 
\end{remark}

In the next lemma, we extract the key features of the property dually contracting. Apart from Proposition 
\ref{prop: non-amenability}, these are the only ones we will need in the proof of the main theorem. So we could 
have stated them as a definition, but Definition \ref{def: dual contraction} is much easier to state and to 
verify.

\begin{lemma}\label{lem: key features}
	Let $X$ be a compact ultrametric space with dually contracting similarity structure $\mathrm{Sim}$. 
	Then there exists a sequence $(S_i)_{i\in\mathbb{N}}$ where each $S_i$ is a set $\{B_i^1,...,B_i^{n_i}\}$ 
	of pairwise disjoint balls in $X$ satisfying the following properties.
	\begin{itemize}
		\item[i)] For each $i,k$ there exists a similarity $X\rightarrow B_i^k$ in $\mathrm{Sim}$.
		\item[ii)] $|S_i|\xrightarrow{i\rightarrow\infty}{}\infty$.
		\item[iii)] For every $i_0\in\mathbb{N}$ and every partition $\mathcal{P}$ of $X$ into non-empty closed 
		open subspaces there is an $i\geq i_0$ such that for every $B\in S_i$ there exists $P\in\mathcal{P}$ with 
		$B\subset P$.
	\end{itemize}
\end{lemma}

\begin{proof}
	Let $B_1^1$ and $B_1^2$ be two disjoint proper subballs of $X$ and $\gamma_i:X\rightarrow B^i_1$ similarities 
	in $\mathrm{Sim}$ for $i=1,2$. We will define the $S_i$'s inductively. First set $S_1=\{B_1^1,B_1^2\}$. Now 
	assume $S_i=\{B_i^1,...,B_i^{n_i}\}$ has been constructed. Then define
	\[S_{i+1}:=\{\gamma_1(B_i^k),\gamma_2(B_i^k)\ |\ 1\leq k\leq n_i\}.\]
	It is clear that $|S_i|=2^i$ so that ii) holds. Using that $\mathrm{Sim}$ is closed under restriction and 
	composition, it is easy to show property i). It is also quite clear that the balls in each $S_i$ are pairwise 
	disjoint. For iii), first define
	\[\mathrm{depth}(S)=\min\{\mathrm{depth}(B)\ |\ B\in S\}\]
	for any finite set $S$ of balls in $X$. Then the claim follows from Lemma \ref{lem: depth and partition} if 
	we show
	\[\lim_{i\rightarrow\infty}\mathrm{depth}(S_i)=\infty.\]
	Note that an application of the contractions $\gamma_1$ or $\gamma_2$ to a ball increases its depth by at 
	least one. It follows that $\mathrm{depth}(S_i)$ increases by at least one if $i$ increases by one and 
	therefore goes to infinity if $i$ tends to infinity.
\end{proof}

\begin{proposition}\label{prop: non-amenability}
	If $\mathrm{Sim}$ is a dually contracting similarity structure, then the local similarity group 
	$\Gamma=\Gamma(\mathrm{Sim})$ contains a non-abelian free subgroup and is therefore non-amenable. 
\end{proposition}

\begin{proof}
	We will identify two elements in $\Gamma$, $a_1$ and $a_2$, with $\mathrm{ord}(a_1)=3$ and 
	$\mathrm{ord}(a_2)=2$. We will also construct disjoint subsets $X_1$ and $X_2$ of $X$ such that
	\begin{eqnarray*}
		a_1X_2 & \subset & X_1\\
		a_1^2X_2 & \subset & X_1\\
		a_2X_1 & \subset & X_2
	\end{eqnarray*}
	Thus the ping-pong lemma will tell us that the subgroups $H_1:=\langle a_1\rangle\cong\mathbb{Z}_3$ and 
	$H_2:=\langle a_2\rangle\cong\mathbb{Z}_2$ together generate a free product in $\Gamma$, 
	i.e.~$\langle H_1,H_2\rangle\cong H_1*H_2\cong\mathbb{Z}_3 * \mathbb{Z}_2$ is a subgroup of $\Gamma$, 
	which itself contains a non-abelian free subgroup. Let's turn to the construction 
	(see Figure~\ref{fig: dually contracting} below). Let $A_1$ and $A_2$ be two disjoint proper subballs of 
	$X$ and $\gamma_i:X\rightarrow A_i$ for $i=1,2$ two similarities in $\mathrm{Sim}$. Set
	\begin{eqnarray*}
		B_1 & := & \gamma_1(A_1)\\
		B_2 & := & \gamma_1(A_2)\\
		B_3 & := & \gamma_2(A_1)\\
		B_4 & := & \gamma_2(A_2)
	\end{eqnarray*}
	These are pairwise disjoint balls in $X$. The similarities $\gamma_1$ and $\gamma_2$ induce similarities 
	between any pair of the balls $A_i$ and $B_i$. For example
	\begin{eqnarray*}
		\delta_2 & := & 
		\gamma_2|_{A_1}\circ\gamma_1\circ\gamma_2^{-1}\circ\gamma_1^{-1}|_{B_2}:B_2\rightarrow B_3\\
		\delta_3 & := &
		\gamma_2|_{A_2}\circ\gamma_2\circ\gamma_1^{-1}\circ\gamma_2^{-1}|_{B_3}:B_3\rightarrow B_4\\
		\delta_4 & := &
		\gamma_1|_{A_2}\circ\gamma_2^{-1}|_{B_4}:B_4\rightarrow B_2
	\end{eqnarray*}
	Now define $a_1$ to be the identity except on the balls $B_2$, $B_3$ and $B_4$ where
	\begin{eqnarray*}
		a_1|_{B_2}&:=&\delta_2: B_2\rightarrow B_3\\
		a_1|_{B_3}&:=&\delta_3: B_3\rightarrow B_4\\
		a_1|_{B_4}&:=&\delta_4: B_4\rightarrow B_2
	\end{eqnarray*}
	It is straightforward to verify
	\begin{eqnarray*}
		\delta_4\circ\delta_3\circ\delta_2 & = & \mathrm{id}_{B_2}\\
		\delta_2\circ\delta_4\circ\delta_3 & = & \mathrm{id}_{B_3}\\
		\delta_3\circ\delta_2\circ\delta_4 & = & \mathrm{id}_{B_4}
	\end{eqnarray*}	
	so that $a_1$ has order $3$. Define $a_2$ to be the identity except on the balls $B_2$ and $A_2$ where
	\begin{eqnarray*}
		a_2|_{B_2}&:=&\gamma_1^{-1}|_{B_2}:B_2\rightarrow A_2\\
		a_2|_{A_2}&:=&\gamma_1|_{A_2}:A_2\rightarrow B_2
	\end{eqnarray*}
	It is trivial to check $a_2^2=\mathrm{id}_X$. Last but not least define $X_1:=A_2$ and $X_2:=B_2$. It 
	is easy to see from the definitions that the relations at the beginning of the proof hold, so it is completed.
\end{proof}

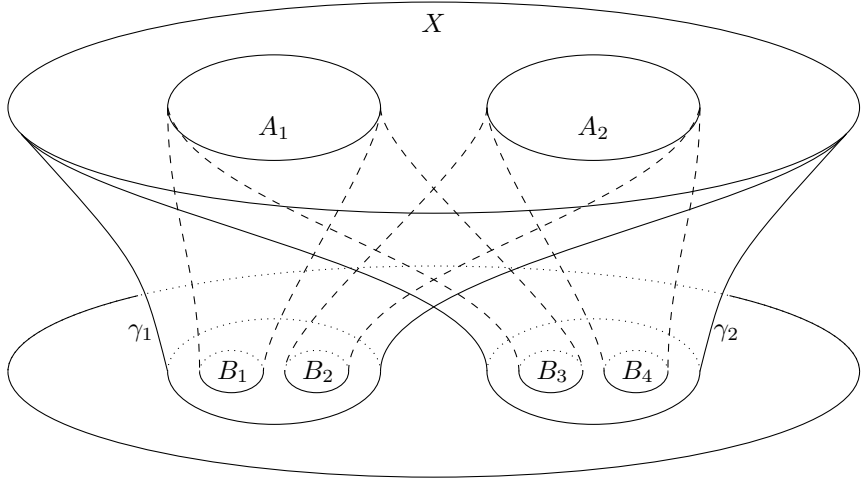
\begin{figure}[H]
\centering
\begin{tikzpicture}[scale=0.7]
	\draw[dotted] (8,-5) arc (0:180:8 and 2);
	\draw (-8,-5) arc (180:360:8 and 2);
	\draw (8,-5) arc (0:46:8 and 2);
	\draw (-8,-5) arc (180:134:8 and 2);
	
	\draw[dotted] (-1,-5) arc (0:180:2 and 1);
	\draw (-5,-5) arc (180:360:2 and 1);
	\draw[dotted] (5,-5) arc (0:180:2 and 1);
	\draw (1,-5) arc (180:360:2 and 1);
	
	\draw[dotted] (-3.2,-5) arc (0:180:0.6 and 0.4) node[right=3pt] {$B_1$};
	\draw (-4.4,-5) arc (180:360:0.6 and 0.4);
	\draw[dotted] (-1.6,-5) arc (0:180:0.6 and 0.4) node[right=3pt] {$B_2$};
	\draw (-2.8,-5) arc (180:360:0.6 and 0.4);
	\draw[dotted] (2.8,-5) arc (0:180:0.6 and 0.4) node[right=3pt] {$B_3$};
	\draw (1.6,-5) arc (180:360:0.6 and 0.4);
	\draw[dotted] (4.4,-5) arc (0:180:0.6 and 0.4) node[right=3pt] {$B_4$};
	\draw (3.2,-5) arc (180:360:0.6 and 0.4);
	
	\draw[dashed] (-5,0) .. controls (-5,-2) and (-4.4,-3) .. (-4.4,-5);
	\draw[dashed] (-1,0) .. controls (-1,-1) and (-3.2,-4) .. (-3.2,-5);	
	
	\draw[dashed] (-5,0) .. controls (-4.5,-2) and (1.6,-3) .. (1.6,-5);	
	\draw[dashed] (-1,0) .. controls (-1,-1) and (2.8,-4) .. (2.8,-5);
	
	\draw[dashed] (1,0) .. controls (1,-1) and (-2.8,-4) .. (-2.8,-5);	
	\draw[dashed] (5,0) .. controls (4.5,-2) and (-1.6,-3) .. (-1.6,-5);
	
	\draw[dashed] (1,0) .. controls (1,-1) and (3.2,-4) .. (3.2,-5);	
	\draw[dashed] (5,0) .. controls (5,-2) and (4.4,-3) .. (4.4,-5);
	
	\draw (0,0) ellipse (8 and 2) node[above=25pt] {$X$};
	\draw (-3,0) ellipse (2 and 1) node[below] {$A_1$};
	\draw (3,0) ellipse (2 and 1) node[below] {$A_2$};	
	
	\draw (-7.75,-0.5) .. controls (-5.5,-3) .. (-5,-5) node[left=10pt,above=8pt] {$\gamma_1$};
	\draw (7.75,-0.5) .. controls (5.5,-3) .. (5,-5) node[right=10pt,above=8pt] {$\gamma_2$};
	\draw (7.75,-0.5) .. controls (6.5,-2) and (-1,-3) .. (-1,-5);
	\draw (-7.75,-0.5) .. controls (-6.5,-2) and (1,-3) .. (1,-5);
\end{tikzpicture}
\vspace{2mm}
\caption{Figure for the proof of Proposition \ref{prop: non-amenability}. The ellipses below are copies of the 
ellipses above and represent the various balls appearing in the proof. The vertical wires represent the 
similarities $\gamma_i$ or restrictions of them.}
\label{fig: dually contracting}
\end{figure}

\section{Proof of the main theorem}\label{sec: proof}

The proof relies on a common feature of relatives of Thompson's groups: they contain products of arbitrarily 
many copies of themselves as subgroups. This feature was utilized in homological vanishing results 
before~\cites{monod, bader-furman-sauer}. 

\subsection{A spectral sequence}
Our main tool will be a spectral sequence explained in Brown's book~\cite{brown-book}*{Chapter~VII.7} which we 
will summarize now. Let $\Gamma$ be a group and $Z$ a simplicial complex with a simplicial $\Gamma$-action. 
Let $M$ be a $\mathbb{Z}[\Gamma]$-module. For $\sigma$ a simplex in $Z$, denote by $\Gamma_\sigma$ the isotropy 
group of $\sigma$, i.e.~all the elements in $\Gamma$ which fix $\sigma$ as a set of vertices. Let $M_\sigma$ be 
the orientation $\mathbb{Z}[\Gamma_\sigma]$-module, i.e.~$M_\sigma=M$ as an abelian group together with the 
action
\[\Gamma_\sigma\times M\rightarrow M\hspace{3mm}(g,m)\mapsto
	\begin{cases}
	gm&\text{if }g\text{ is an even permutation of the vertices of }\sigma\\
	-gm&\text{if }g\text{ is an odd permutation of the vertices of }\sigma
	\end{cases}\]
Furthermore, let $\Sigma_p$ be a set of $p$-cells representing the $\Gamma$-orbits of $Z$. Then there is a 
spectral sequence $E^k_{pq}$ with first term 
\[E^1_{pq}=\bigoplus_{\sigma\in\Sigma_p}H_q(\Gamma_\sigma,M_\sigma)\Rightarrow H^\Gamma_{p+q}(Z,M)\]
converging to the $\Gamma$-equivariant homology of $Z$ with coefficients in $M$. In our case, $Z$ will be acyclic, 
so that $H^\Gamma_{p+q}(Z,M)=H_{p+q}(\Gamma,M)$. Furthermore, $\Gamma_\sigma$ will fix $\sigma$ vertex-wise, so 
that $M_\sigma=M$ as $\mathbb{Z}[\Gamma_\sigma]$-modules. We therefore have a spectral sequence $E^k_{pq}$ with
\[E^1_{pq}=\bigoplus_{\sigma\in\Sigma_p}H_q(\Gamma_\sigma,M)\Rightarrow H_{p+q}(\Gamma,M).\]

\subsection{The poset of partitions into closed open sets}
Now let $\Gamma=\Gamma(\mathrm{Sim})$ be a local similarity group coming from a dually contracting similarity 
structure $\mathrm{Sim}$ on the compact ultrametric space $X$. Next we define, for each $n\in\mathbb{N}$, a 
simplicial $\Gamma$-complex $Z_n$ associated to a poset $(\mathbf{P}_n,\le)$. Unlike the simplicial complex 
in~\cite{farley-hughes}, used for proving finiteness properties, it has large isotropy groups. By definition, 
an element in $\mathbf{P}_n$ is a set (\emph{partition}) $\mathcal{P}=\{P_1,...,P_k\}$ of pairwise disjoint 
non-empty closed open subspaces of $X$ with $X=P_1\cup...\cup P_k$ satisfying the following extra condition: 
There are at least $n$ elements contained in $\mathcal{P}$ which are locally $\mathrm{Sim}$-equivalent to $X$.

By Lemma \ref{lem: key features}, $\mathbf{P}_n\ne\emptyset$. Let $\mathcal{P},\mathcal{Q}\in\mathbf{P}_n$. 
We say $\mathcal{P}\leq\mathcal{Q}$ iff $\mathcal{Q}$ refines $\mathcal{P}$, that is, 
$\forall_{Q\in\mathcal{Q}}\exists_{P\in\mathcal{P}}~Q\subset P$. Then $(\mathbf P_n,\le)$ is a poset. 
Explicitly, a simplex in the associated simplicial complex $Z_n$ is a finite set of vertices which can be 
totally ordered using the partial order on $\mathbf{P}_n$. We write $\{\mathcal{P}_1<...<\mathcal{P}_k\}$ 
for such a $(k-1)$-simplex.

Next we will show that $(\mathbf P_n,\le)$ is directed, which implies that $Z_n$ is contractible. So let 
$\mathcal{P},\mathcal{Q}\in\mathbf P_n$. First, it is easy to see that there is a partition $\mathcal{R}$ 
into open and closed sets which refines both $\mathcal{P}$ and $\mathcal{Q}$. But we have to refine it even 
more so that it satisfies the extra condition. From Lemma \ref{lem: key features} part iii) and ii) we obtain 
that there are at least $n$ disjoint balls $B_1,...,B_s$ such that every ball $B_i$ is contained in some element 
of $\mathcal{R}$ and from part i) we know that every ball $B_i$ is $\mathrm{Sim}$-equivalent to $X$. So we can 
take these balls as elements of a refinement of $\mathcal{R}$.

We endow $Z_n$ with the simplicial $\Gamma$-action 
\[g\{P_1,...,P_k\}=\{g(P_1),...,g(P_k)\}\]
for a vertex $\mathcal{P}=\{P_1,...,P_k\}$ and $g\in \Gamma$. We have $g\mathcal{P}\leq g\mathcal{Q}$ whenever 
$\mathcal{P}\leq\mathcal{Q}$. It follows that the action is indeed simplicial and that whenever $g\in\Gamma$ 
fixes a simplex as a set of vertices, then it fixes it vertex-wise.

We want to take a closer look at the isotropy groups $\Gamma_\sigma$ for 
$\sigma=\{\mathcal{P}_1<...<\mathcal{P}_k\}$ a simplex. First consider the case $k=1$. Write 
$\mathcal{P}=\mathcal{P}_1$. If $g\in\Gamma_\sigma$ then $g\mathcal{P}=\mathcal{P}$ and consequently, there 
is a permutation $\pi$ of the set $\mathcal{P}$ such that $g(P)=\pi(P)$ for every $P\in\mathcal{P}$. Write 
$\Sigma_\mathcal{P}$ for the group of permutations of the set $\mathcal{P}$. We therefore have
\[\Gamma_\sigma=\big\{g\in\Gamma\ |\ \exists_{\pi\in\Sigma_\mathcal{P}}\forall_{P\in\mathcal{P}}~g(P)=
	\pi(P)\big\}.\]
Now let $k\geq 1$ be arbitrary. In this case, we have $g\mathcal{P}_i=\mathcal{P}_i$ for each $i=1,...,k$. 
First we start with a preliminary remark. Let $\mathcal{P}\leq\mathcal{Q}$ be two vertices. Then there is a 
unique function $f:\mathcal{Q}\rightarrow\mathcal{P}$ such that $Q\subset f(Q)$ for all $Q\in\mathcal{Q}$. Let 
$\pi\in\Sigma_\mathcal{Q}$. Then $\pi$ is called $\mathcal{P}$-admissible iff there is a 
$\rho\in\Sigma_\mathcal{P}$ such that for all $P\in\mathcal{P}$ and all $Q\in f^{-1}(P)$ we have 
$f(\pi(Q))=\rho(P)$. In other words, $\pi$ is a permutation of $\mathcal{Q}$ which gives a permutation of 
$\mathcal{P}$ when we write each element in $\mathcal{P}$ as a disjoint union of elements in $\mathcal{Q}$. 
The set of all $\mathcal{P}$-admissible elements forms a subgroup of $\Sigma_\mathcal{Q}$, denoted by 
$\Sigma_{\mathcal{P}\leq\mathcal{Q}}$. More generally, if we have an ascending chain 
$\mathcal{Q}_1\leq...\leq\mathcal{Q}_l$ of vertices, then we can define the subgroup 
$\Sigma_{\mathcal{Q}_1\leq...\leq\mathcal{Q}_l}$ of $\Sigma_{\mathcal{Q}_l}$ consisting of all elements in 
$\Sigma_{\mathcal{Q}_l}$ which are $\mathcal{Q}_i$-admissible for all $i=1,...,l-1$. In particular, we have 
defined a subgroup $\Sigma_\sigma$ of $\Sigma_{\mathcal{P}_k}$ for the simplex 
$\sigma=\{\mathcal{P}_1<...<\mathcal{P}_k\}$ from above. This group is defined exactly in a way such that
\[\Gamma_\sigma=\big\{g\in\Gamma\ |\ \exists_{\pi\in\Sigma_\sigma}\forall_{P\in\mathcal{P}_k}~g(P)=\pi(P)\big\}.\]
The group
\[\Lambda_\sigma=\big\{g\in\Gamma\ |\ \forall_{P\in\mathcal{P}_k}~g(P)=P\big\}\cong
	\prod_{P\in\mathcal{P}_k}\Gamma(\mathrm{Sim}|_P)\]
is a normal subgroup of $\Gamma_\sigma$. It is also of finite index in $\Gamma_\sigma$ because the quotient 
$\Gamma_\sigma/\Lambda_\sigma$ injects into the finite group $\Sigma_\sigma$.

\subsection{K\"unneth theorems}
The following K\"unneth vanishing result was proved in the context of Farber's extended $l^2$-homology 
in~\cite{farber-weinberger}*{3.~Appendix}. For the proof of the exact formulation below 
see~\cite{lück-book}*{Lemma~12.11 on p.~448}.  

\begin{proposition}
	Let $G=G_1\times G_2$ be a product of two groups. Assume that $H_p(G_1,\mathcal{N}(G_1))=0$ for $p\leq n_1$ 
	and $H_p(G_2,\mathcal{N}(G_2))=0$ for $p\leq n_2$. Then we have $H_p(G,\mathcal{N}(G))=0$ for 
	$p\leq n_1+n_2+1$. Note that the case $n_i=-1$ is allowed and gives a non-trivial statement.
\end{proposition}

\begin{corollary}\label{cor: künneth}
	Let $m\geq n\geq 2$ and $G=G_1\times...\times G_m$ be a product of $m$ groups. Assume that 
	$H_0(\bullet,\mathcal{N}(\bullet))$ vanishes for at least $n$ of the groups $G_i$. Then 
	$H_*(G,\mathcal{N}(G))$ vanishes up to degree $n-1$.
\end{corollary}

We also need cohomological versions of these results with coefficients in the group ring. 

\begin{proposition}
	Let $G=G_1\times G_2$ be a product of two groups of type $FP_\infty$. Assume that $H^p(G_1,\bbZ [G_1])=0$ 
	for $p\leq n_1$ and $H^p(G_2,\bbZ[ G_2])=0$ for $p\leq n_2$. Then we have $H^p(G,\bbZ [G])=0$ for 
	$p\leq n_1+n_2+1$. 
\end{proposition}

\begin{proof}
	Let $P_\ast$ be a $\bbZ[G_1]$-resolution of $\bbZ$ such that each $P_i$ is a finitely generated free 
	$\bbZ [G_1]$-module. Let $Q_\ast$ be a similar resolution for $G_2$. Then
	\[C^\ast=\hom_{\bbZ [G_1]}(P_\ast, \bbZ [G_1])\hspace{4mm}\text{and}\hspace{4mm}D^\ast=
		\hom_{\bbZ [G_2]}(Q_\ast, \bbZ [G_2])\]
	are cochain complexes of free abelian groups which compute $H^*(G_1,\bbZ [G_1])$ and $H^*(G_2,\bbZ[ G_2])$ 
	respectively. For $\bbZ[G_i]$-modules $M_i$, $i\in\{1,2\}$, the cochain cross 
	product~\cite{brown-book}*{Chapter~V.2} 
	\begin{equation}\label{eq: cochain cross}
		\hom_{\bbZ[G_1]}(P_\ast, M_1)\otimes_\bbZ \hom_{\bbZ[G_2]}(Q_\ast, M_2)\to 
		\hom_{\bbZ[G]}(P_\ast\otimes_\bbZ Q_\ast, M_1\otimes_\bbZ M_2)
	\end{equation}
	is an isomorphism of cochain complexes since all $P_i$ and $Q_j$ are finitely generated free. If 
	$M_i=\bbZ[G_i]$, then $M_1\otimes_\bbZ M_2\cong \bbZ[G]$ as $\mathbb{Z}[G]$-modules and the right hand 
	side of~\eqref{eq: cochain cross} computes 
	$H^\ast(G, \bbZ [G])$~\cite{brown-book}*{Proposition~(1.1) on p.~107}. By a suitable 
	K\"unneth theorem~\cite{dold}*{Theorem~9.13 on p.~164} and the fact that $C^\ast, D^\ast$ are free as 
	$\bbZ$-modules, the homology of $C^\ast\otimes_\bbZ D^\ast$ vanishes in degrees $\le n_1+n_2+1$. 
\end{proof}

The following corollary follows from $H^0(G,\bbZ[G])\cong \bbZ[G]^G=0$ for infinite $G$.

\begin{corollary}\label{cor: cohomological kuenneth}
	Let $G=G_1\times...\times G_n$ be a product of $n$ infinite $FP_\infty$ groups. Then $H^*(G,\bbZ[G])$ 
	vanishes up to degree $n-1$.
\end{corollary}

\subsection{Proof of Theorem~\ref{thm: main thm}}
Let $n\geq 2$ be arbitrary. Consider the simplicial $\Gamma$-complex $Z_n$ from above. From the discussion it 
follows that we have a spectral sequence $E^k_{pq}$ with
\begin{equation}\label{eq: proof main thm}
	E^1_{pq}=\bigoplus_{\sigma\in\Sigma_p}H_q(\Gamma_\sigma,\mathcal{N}(\Gamma))\Rightarrow 
	H_{p+q}(\Gamma,\mathcal{N}(\Gamma)).
\end{equation}
Fix a simplex $\sigma=\{\mathcal{P}_1<...<\mathcal{P}_p\}$. First observe the group 
$\Lambda_\sigma\cong\prod_{P\in\mathcal{P}_p}\Gamma(\mathrm{Sim}|_P)$ defined above. From the extra condition 
on the vertices we know that at least $n$ elements of $\mathcal{P}_p$ are locally $\mathrm{Sim}$-equivalent to 
$X$ and therefore, by Lemma \ref{lem: restricted subgroups}, at least $n$ of the groups 
$\Gamma(\mathrm{Sim}|_{P})$ with $P\in\mathcal{P}_p$ are isomorphic to $\Gamma$. $\Gamma$ is infinite by 
Proposition \ref{prop: finiteness} and non-amenable by Proposition \ref{prop: non-amenability}. Going back to a 
result by Kesten~\cite{lück-book}*{Lemma~6.36 on p.~258}, this is equivalent to 
$H_0(\Gamma,\mathcal{N}(\Gamma))=0$. By Corollary \ref{cor: künneth} we therefore have 
$H_q(\Lambda_\sigma,\mathcal{N}(\Lambda_\sigma))=0$ for $q=0,...,n-1$. Since $\Lambda_\sigma$ is normal in 
$\Gamma_\sigma$ we have $H_q(\Gamma_\sigma,\mathcal{N}(\Gamma_\sigma))=0$ for $q=0,...,n-1$ by 
\cite{lück-book}*{Lemma~12.11}. Since $\mathcal{N}(\Gamma)$ is a flat ring extension of 
$\mathcal{N}(\Gamma_\sigma)$~\cite{lück-book}*{Theorem~6.29}, it follows that 
\[H_q(\Gamma_\sigma,\mathcal{N}(\Gamma))=0~~\text{for $q\in\{0,...,n-1\}$}.\]
Consequently, the spectral sequence \eqref{eq: proof main thm} collapses except possibly in the region 
$p\geq0$, $q\geq n$ and therefore
\[H_i(\Gamma,\mathcal{N}(\Gamma))=0~~\text{for $i\leq n-1$}.\]
Because $n$ is arbitrary, Theorem~\ref{thm: main thm} follows.

\subsection{Proof of Theorem~\ref{thm: byproduct thm}} 
The proof is similar to the one above and we only describe the necessary modifications. Hughes and Farley 
proved that $\Gamma$ is of type $F_\infty$ (which implies type $FP_\infty$) under the assumptions on 
$\mathrm{Sim}$~\cite{farley-hughes}*{Theorem~1.1} and it is infinite because of Proposition 
\ref{prop: finiteness}. Instead of~\eqref{eq: proof main thm}, we use the cohomological version of Brown's 
spectral sequence with coefficients in the group ring: 
\[E_1^{pq}=\prod_{\sigma\in\Sigma_p} H^q(\Gamma_\sigma, \bbZ[\Gamma])\Rightarrow H^{p+q}(\Gamma, \bbZ[\Gamma]).\]
Write $\mathcal{P}_p=\{P_1,...,P_k\}$ such that the first $n$ elements are locally $\mathrm{Sim}$-equivalent to 
$X$. Observe the normal subgroup
\[\Lambda_\sigma'=\prod_{i=1}^n\Gamma(\mathrm{Sim}|_{P_i})\vartriangleleft\Lambda_\sigma=
	\prod_{i=1}^k\Gamma(\mathrm{Sim}|_{P_i}).\]
By Corollary~\ref{cor: cohomological kuenneth} we obtain $H^q(\Lambda_\sigma',\mathbb{Z}[\Lambda_\sigma'])=0$ 
for $q\in\{0,\ldots, n-1\}$. Since $\bbZ[\Gamma]$ is a free $\bbZ[\Lambda_\sigma']$-module and group cohomology 
of $FP_\infty$-groups commutes with direct limits in the coefficients~\cite{brown-book}*{Theorem~(4.8) on p.~196}, 
we obtain $H^q\bigl(\Lambda_\sigma',\bbZ[\Gamma]\bigr)=0$ for $q\in\{0,\ldots, n-1\}$. Now an application of the 
cohomological Hochschild-Lyndon-Serre spectral sequence to the group extension 
\[1\to\Lambda_\sigma'\to \Lambda_\sigma\to \Lambda_\sigma/\Lambda_\sigma'\to 1\]
and the coefficient module $\bbZ[\Gamma]$ yields $H^q(\Lambda_\sigma,\bbZ[\Gamma])=0$ for $q\in\{0,\ldots, n-1\}$. 
Apply this spectral sequence once more to the group extension
\[1\to\Lambda_\sigma\to\Gamma_\sigma\to\Gamma_\sigma/\Lambda_\sigma\to 1\]
to obtain
\[H^q(\Gamma_\sigma, \bbZ[\Gamma])=0~~\text{for $q\in\{0,\ldots, n-1\}$}.\]
Now proceed as above.

\end{document}